\documentclass{article}
\usepackage{geometry}

\usepackage{authblk}
\usepackage{type1cm}        
\usepackage{makeidx}         
\usepackage{graphicx}        
\usepackage{multicol}        
\usepackage[bottom]{footmisc}

\usepackage{cite}
\usepackage{amsmath,mathrsfs}
\usepackage{algorithmic}
\usepackage{graphicx}
\usepackage{textcomp}
\usepackage{lipsum}
\usepackage{stfloats}
\PassOptionsToPackage{dvipsnames,svgnames,table}{xcolor}

\usepackage{microtype}
\usepackage{multicol}

\usepackage{bigints}
\usepackage{mathrsfs}
\usepackage{mathtools,amsthm,amssymb}
\usepackage{standalone}

\usepackage{color}

\usepackage{multirow}
\usepackage{upgreek}
\usepackage{accents}

\newcommand{\reals}{\mathbb{R}}

\newcommand{\A}{\mathrm{A}}
\newcommand{\B}{\mathrm{B}}
\newcommand{\Dmatrix}{\mathrm{D}}
\newcommand{\K}{\mathrm{K}}

\newcommand{\Ku}{\mathrm{K}}
\newcommand{\C}{\mathrm{C}}

\newcommand{\G}{\mathrm{G}}

\newcommand{\J}{\mathrm{J}}
\newcommand{\Lm}{\mathrm{L}}
\newcommand{\V}{\mathrm{V}}
\newcommand{\M}{\mathrm{M}}
\newcommand{\Mp}{\mathrm{M}^{\dagger}}
\newcommand{\N}{\mathrm{N}}
\newcommand{\Pp}{\mathrm{P}}
\newcommand{\W}{\mathrm{W}}
\newcommand{\Y}{\mathrm{Y}}
\newcommand{\Z}{\mathrm{Z}}
\newcommand{\Tm}{\mathrm{T}}
\newcommand{\Ss}{\mathrm{S}}
\newcommand{\R}{\mathrm{R}}
\newcommand{\Ematrix}{\mathrm{E}}

\newcommand{\Lap}{\mathcal{L}}

\newcommand{\Imm}{\mathrm{I}}

\newcommand{\T}{\top}
\newtheorem{theorem}{Theorem}

\newtheorem{lem}{Lemma}
\newtheorem{assump}{Assumption}
\newtheorem{remark}{Remark}
\newtheorem{prop}{Proposition}

\newcommand*\xbar[1]{%
   \hbox{%
     \vbox{%
       \hrule height 0.7pt 
       \kern0.35ex
       \hbox{%
         \kern-0.1em
         \ensuremath{#1}%
         \kern-0.1em
       }%
     }%
   }%
}

\title{Dynamic allocation function design in the presence of magnitude saturating inputs}

\author[1]{Thiago Alves Lima}
\author[2]{Sophie Tarbouriech}
\affil[1]{Université Paris-Saclay, CNRS, CentraleSupélec, Laboratoire des Signaux et Systèmes, 91190, Gif-sur-Yvette, France}
\affil[2]{LAAS-CNRS, Université de Toulouse, CNRS, Toulouse, France}

\begin{document}
\date{}
\maketitle

\abstract{This chapter deals with the design of dynamic allocation functions for systems with saturating actuators. The goal of the allocator consists in redistributing the desired control effort within the multiple actuators by penalizing each actuator to be more or less used, while also taking into account a criterion for minimization of their total energy consumption over time. Anti-windup gains are added to both the controller and the dynamic allocator  
to deal with the saturation condition. Two cases are considered: the plant is affected by bounded disturbance and the influence matrix is supposed to be affected by uncertainty. Convex conditions for the co-design of both the dynamic allocator and anti-windup gains are then expressed in the form of linear matrix inequalities (LMIs). Such conditions allow to deal with the multiple objective problems of enlarging the estimates of the basin of attraction,  minimizing the total energy consumption of the actuators and maximizing the size of the admissible disturbance. The satellite formation problem borrowed from the literature is revised to illustrate the proposed technique and show its effectiveness in both cases (perturbed system and robust case).}

\section{Introduction}
\label{intro}

Control allocation is an essential issue when dealing with over-actuated systems and aims at applying some dedicated algorithm to distribute the computed control effort throughout the multiple actuators that together drive the plant states and/or outputs. Numerous applications embed the control allocation problem, for example, when the plant includes torques and/or forces as inputs generated by a set of multiple actuators that jointly produce the desired control effort: a classical example corresponds to microthrusters in space applications.  The advantages of the control allocation approach are modularity and ability to handle constraints (see, for example, \cite{JOHANSEN_2013,Johansen2008}).

Several papers deal with the control allocation problem from a specific application point of view, in particular in the aeronautical or spatial contexts: see, for example,  \cite{Jin2005}, \cite{Oppenheimer2006}, \cite{Boada2013}, \cite{Durham2016}. Of course, the literature also addresses technical solutions from a theoretical point of view. For example, in  \cite{GALEANI2015} and \cite{Serrani2012}  the output regulation problem of over-actuated systems in the presence of full information regarding the system states and exogenous inputs is studied. In particular, \cite{GALEANI2015} uses the hybrid systems framework to propose an allocation mechanism which takes into account for input constraints. In \cite{Petersen_2006}, given the actuators constraints, optimal allocation are computed from optimization-based algorithms, as interior point method. The online implementation of this kind of technique can be, however, computationally expensive, while stability analysis of the closed loop is not straightforward. Indeed, the presence of constraints induces errors between the desired control effort and the actual plant input, which may lead to poor response and even instability. To overcome this, Lyapunov-based approaches with guarantees of stability for the constrained closed-loop system have also been studied, for example, in \cite{Castro_2019,Benosman_2009,Liao_2007,Johansen_2004}. In the same context, in \cite{MTNS2021} the authors presented allocation function and anti-windup design with an optimization procedure in order to minimize this error while guaranteeing closed-loop stability. The proposed allocation format and design procedure did not however take into account the ability to penalize the use of the different actuators and energy consumption~minimization.

In \cite{ZACCARIAN_2009}, the concepts of weak and strong redundancy are formally defined, implying for the first concept that multiple actuators can induce the same steady-state value for the plant output while the second one implies that they can also impose equal trajectories. The paper focused on using a dynamic allocator system between the controller and the plant to distribute the control effort by penalizing the use of multiple actuators. The cases of input and rate saturation of the actuators were also addressed. The use of dynamic allocation functions was then shown to be a good alternative in terms of both computational effort and robustness. \textcolor{black}{In \cite{AlvesLima2021}, the idea of using dynamic allocator was extended to a setup where: i) the co-design of the anti-windup and allocation function is done simultaneously through LMIs that are shown to be always feasible in the local case and ii) the size of the plant input (noted $m_c$ here) may be different to the size of the allocator output (noted $m_a$ here).}

\textcolor{black}{In the current chapter, we follow the same line as in \cite{AlvesLima2021} but with the ambition to deal with an even more general setup than the above-described problems. In this sense, we consider that: iii) the plant can be perturbed by an additive disturbance, which is bounded; and iv) the influence matrix (noted $\M$ here), which maps how the different actuators generate the plant input, can be affected by \emph{unknown} parameters}. Theoretical conditions based on Lyapunov theory
 are formulated in terms of linear matrix inequalities (LMIs) in order to solve the co-design of the allocator and anti-windup loop. From these conditions,  multiple objective problems can be handled: in order to enlarge the estimates of the basin of attraction,  minimize the total energy consumption of the actuators and maximize the size of the admissible disturbance. The satellite formation
problem borrowed from the literature is then revised to illustrate the proposed technique and show its effectiveness in the both cases of perturbed system and uncertain influence matrix  case.

The chapter is organized as follows. Section \ref{sec:pb} is dedicated to present the general view of the control allocation, and to specify the class of the plant, controller and allocation function under consideration. Section \ref{sec:prelim} presents some useful results for developing the main conditions. Section \ref{sec:main} then proposes the main theoretical conditions, in the perturbed and robust case. The associated optimization schemes are discussed in Section \ref{sec:optim}. In Section \ref{sec:simu}, The satellite formation
problem borrowed from the literature emphasizes the interest of the proposed approach. Finally, Section \ref{sec:conclu} ends the chapter by presenting concluding remarks and forthcoming issues.\\

\noindent \textbf{Notation.} For a matrix $\Y $ in $ \reals^{n \times m}$, $\Y^{\T} $ in $ \reals^{m \times n}$ means its transpose, $\Y_{\left(i\right)}$ denotes its $i$th row, while for $v $ in $ \reals^{m}$, $v_{\left(i\right)}$ denotes its $i$th component. For matrices $\W = \W^{\T}$ and $\Z = \Z^{\T} $ in $ \reals^{n \times n}$, $\W \succ \Z$ means that $\W-\Z$ is positive definite. Likewise, $\W \succeq \Z$ means that $\W-\Z$ is positive semi-definite. $\mathbb{S}_n^{+}$ stands for the set of symmetric positive definite matrices in $\reals^{n \times n}$. $\Imm$ and $0$ denote identity and null matrices of appropriate dimensions, although their dimensions can be explicitly presented whenever relevant. The $\star$ in $\begin{bmatrix} \A & \B \\ \star & \C \end{bmatrix}$ denotes symmetric blocks, that is $\star = \B^{\T}$. We define He$\{ \A \}$ as the operator He$\{ \A \} = \A+\A^{\T}$. Finally, for matrices $\W$ and $\Z$, diag$(\W,\Z)$ corresponds to the block-diagonal matrix.

\section{Problem formulation}
\label{sec:pb}
\subsection{General view}

Fig. \ref{Fig1} presents the general view of the control allocation problem with anti-windup, where the regulatory case, i.e. $r=0$, is considered. 


\begin{figure}[!hhh]
\center{\includegraphics[width=0.8\linewidth,angle=0]{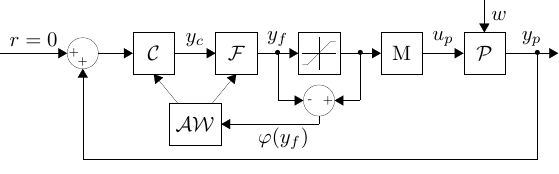}
\caption{General view of control allocation problem with anti-windup.}
\label{Fig1}}
\end{figure}

Subsystems $\mathcal{C}$, $\mathcal{F}$, and $\mathcal{P}$ are the controller, the control allocation device, and the plant, respectively, while $\M$ is the influence matrix and $\mathcal{AW}$ represents some anti-windup strategy to be specified later. The plant is driven by the input $u_p : \reals_{+} \rightarrow \reals^{m_c}$. The controller computes a set of desired efforts that must be injected in the plant in ideal conditions, represented by its output $y_c: \reals_{+} \rightarrow \reals^{m_c}$. The plant input is generated by a set of $m_a \geq m_c$ actuators, represented by the signal $y_f: \reals_{+} \rightarrow  \reals^{m_a}$, and is also affected by external disturbances $w: \reals_{+} \rightarrow  \reals^{n_w}$. The plant input, at each time instant $t \in \reals_{+}$, is given by $u_p (t)=\M sat (y_f (t))$, where for a vector $s \in \reals^{m_a}$, the decentralized saturation function being defined as
\begin{equation}\label{eq_sat}
sat(s_{(i)}) = sign(s_{(i)}) \min\{|s_{(i)}|,\xbar{u}_{(i)} \}, \xbar{u}_{(i)}>0, 
  \end{equation}
\noindent for $i=1, \dots, m_a$, where $\xbar{u}_{(i)}$ denotes the magnitude bound in each actuator.

The presumption of possessing a flawless model of a physical system is frequently an oversight in control engineering, given that real systems typically exhibit intricate dynamics that are not completely captured by simpler models. Disregarding this reality can give rise to consequential issues, such as potential instability or degradation in performance in practical applications. Within the domain of control allocation, a noteworthy wellspring of uncertainty lies in the knowledge of the influence matrix, susceptible to errors arising from the aging and failure of actuators. In this work, we consider that the influence matrix is given as
\begin{equation}\label{eq:uncertainInfluence}
    \M(\theta) = \M_n + \M_u(\theta),
\end{equation}
where $\theta \in \Theta \subset \reals^r$ is an uncertain bounded vector parameter lying in a convex polytope $\Theta$, and $\M_n$ represents the known part of $\M$. The indices $n$ and $u$ in $\M$ refer to ``nominal'' and ``uncertain'', respectively. The following assumption is taken on $\M_u(\theta)$. 

\begin{assump}\label{assumption1}
There exist parameters $\alpha_i(\theta)$, $i \in \{1,..,n_{\alpha}\}$, with \(n_{\alpha}=2^r\), belonging to the unit simplex
\begin{eqnarray} \label{eq:polytope} 
\hspace{-0.2cm}&\Omega=\{\alpha(\theta) \in \mathbb{R}^{n_{\alpha}}:\sum \limits_{i=1}^{n_{\alpha}} \alpha_i=1; \alpha_i \geq 0; i=1,\ldots,n_{\alpha}\},
\end{eqnarray}
such that $\M_u(\theta)$ can be written as the convex combination of $n_{\alpha}$ matrices $M_u$, i.e,  $\M_u(\theta)=\sum_{i=1}^{n_{\alpha}} \alpha_i \M_i$.
\end{assump}

In particular, Assumption~\ref{assumption1} holds if and only if $\M_u(\theta)$ has affine dependence in $\theta$~\cite{boyd2004convex}. The influence matrix $\M$ maps how each individual effort of the $m_a$ actuators combines to generate the inputs acting on the plant. The simplest allocation function often considered in the literature is given by the right pseudo-inverse of $\M$, that is, $\mathcal{F} = \Mp$, with $\M \Mp = \Imm$. In the uncertain case considered in this chapter, such a strategy cannot be implemented since the parameter $\theta$ is unknown. One possibility in this case would be to consider the right-pseudo inverse of the known part of $\M$, with $\M_n \Mp = \Imm$.

In the case the actuator is neither subject to saturation nor uncertainties $\theta$, this allocator can guarantee the stability of the closed loop since the interconnection is given by $u_p = \M_n \Mp_n y_c = y_c$, and therefore no error between $u_p$ and $y_c$ is produced. However, this allocator does not fully take advantage of the multi-actuated nature of the system for energy distribution among actuators, besides lacking stability proofs in the presence of actuator nonlinearities.

\subsection{Plant and controller description}
Consider the plant $\mathcal{P}$ described by the following equations
\begin{equation} \label{eq:model_disturb}
\mathcal{P}_w \sim \begin{cases}
    \Dot{x}_p(t) = \A_p x_p(t) + \B_p u_p(t) + \B_w w(t) , \\
    y_p(t) = \C_p x_p(t).
\end{cases}
\end{equation}
\noindent where $x_p(t) $ in $ \reals ^{n_p}$ is the plant state vector, $u_p(t) $ in $ \reals^{m_c}$ is the plant input, $y_p(t) $ in $ \reals^{q}$ is the measured output, \textcolor{black}{and $w(t) \in \reals^{n_w}$ is an unknown external disturbance}. $\A_p$, $\B_p$, and $\C_p$ are all constant and known matrices of appropriate dimensions. Furthermore, the pairs $(\A_p,\B_p)$ and $(\C_p,\A_p)$ are supposed to be controllable and observable.

In control systems, it is common for the controlled system to be affected by external disturbances of an unknown nature. In order to consider their effects, we can suppose that the disturbance belongs to a set of functions, such as energy-limited or amplitude-bounded functions. In this work, we consider the former type. A review of how to consider the amplitude-bounded disturbances can be found, for example, in \cite[Chapter 2.3.1]{Tarbouriech_2011}.

Consider then the following class of functions 
\begin{equation}\label{def:Wset}
    \mathcal{W} = \{ w : [0,\infty) \rightarrow \reals^{n_w} ; \bigintssss_0^{\infty} w(\tau)^\top \R w (\tau) d \tau < \sigma^{-1} \}
\end{equation}
for some $\sigma>0$ and positive definite matrix $R$.

Let us assume that the \emph{undisturbed nominal} plant (\ref{eq:model_disturb}) (with $w=0$ and $\theta=0$) is stabilized by a dynamic output controller $\mathcal{C}$ linearly designed via the connection $u_p(t)=y_c(t)$, that is without taking into account the saturation and with $\mathcal{F}_n=\Mp_n$. The controller $\mathcal{C}$ is defined by the following equations
\begin{equation} \label{eq:control}
\mathcal{C} \sim \begin{cases}
    \Dot{x}_c(t) = \A_c x_c(t) + \B_c y_p(t) + \Ematrix_c\varphi(y_f(t)), \\
    y_c(t) = \C_c x_c(t) + \Dmatrix_c y_p(t),
\end{cases}
\end{equation}
\noindent where $x_c(t)$ in $ \reals ^{n_c}$ is the controller state vector and $y_c(t)$ in $ \reals^{m_c}$ is the controller output. $\A_c$, $\B_c$, $\C_c$ and $\Dmatrix_c$ are supposed known.
In this chapter, we consider the anti-windup signal $v_{aw}=\Ematrix_c\varphi(y_f)$, with matrix $\Ematrix_c $ in $ \reals ^{n_c \times m_a}$ and deadzone nonlinearity $\varphi(y_f)$ given by
\begin{equation}
\varphi(y_f)=sat(y_f)-y_f,
\label{DZident}
\end{equation}
where the saturation map is defined from (\ref{eq_sat}) and $y_f$ is the output of the allocation function \textcolor{black}{to be defined in the sequence}. Such an anti-windup compensation is added in order to mitigate the undesired effects of saturation  (see, for example, \cite{Tarbouriech_2011}, \cite{bookZaccarian}).

\begin{remark} By construction, the linear connection plant-controller is supposed to be stable. In other words, the controller \eqref{eq:control} (with $v_{aw}(t) = 0, \forall t \in \reals_{+}$) stabilizes the \emph{undisturbed nominal plant} (i.e., \eqref{eq:model_disturb} with $w(t)=0,  \forall t \in \reals_{+}$ and $\theta=0$) through the linear 
interconnection $u_p = y_c$  
and therefore the matrix
\begin{equation}
\A_0 = \left[\begin{array}{cc}
\A_p + \B_p \Dmatrix_c \C_p & \B_p \C_c\\
\B_c \C_p & \A_c
\end{array}\right] \text{ in } \reals^{(n_p+n_c)\times(n_p+n_c)}
\label{eq:A0}
\end{equation}
is Hurwitz.
\label{rem:linearcon}
\end{remark}

\subsection{Dynamic allocation function description}

Consider the influence matrix $\M(\theta)$ in $ \reals ^{m_c \times m_a}$ in the case $m_a > m_c$, supposedly full row rank. Let $\N$ in $ \reals ^{m_a \times n_f}$, $n_f=m_a-m_c$, be a basis for the Kernel of $\M_n$ (the ``nominal'' part of $\M(\theta)$), i.e. $\M_n\N=0$, and $\Mp$ be the right pseudo-inverse of $\M_n$. We then propose the following dynamic allocation function
\begin{equation} \label{eq:alloc}
\hspace{-0.012cm}\mathcal{F} \hspace{-0.08cm}\sim\hspace{-0.08cm} \begin{cases}
     \Dot{x}_f(t) = \K_f \N^{\T}  \W\N x_f(t) \hspace{-0.02cm} + \hspace{-0.02cm}\K_f \N^{\T}\W\Mp y_c(t) \hspace{-0.02cm}+\hspace{-0.02cm} \Ematrix_f\varphi(y_f(t)), \\
     y_f(t) = \N x_f(t) + \Mp y_c(t),
\end{cases}
\end{equation}
\noindent where $x_f(t)$ in $ \reals^{n_f}$ is the allocator state vector, and $y_f(t)$ in $\reals^{m_a}$ is the allocator output. $\W$=diag$(w_1,w_2,\dots,w_{m_a})$ in $\mathbb{S}_{m_a}^{+}$ is a matrix which receives the weightings that penalizes the use of each actuator. Matrices $\K_f$ in $ \reals^{n_f\times n_f}$ and $\Ematrix_f$ in $ \reals^{n_f\times m_a}$ must be designed to achieve desired behavior of the allocator by taking into account the presence of saturation. This allocation format is particularly interesting since it is in some sense optimal in terms of both the allocation error and actuators usage, as explained in the next two remarks \cite{AlvesLima2021}. 

\begin{remark}
Consider the general expression $y_f(t)=\C_f x_f(t) + \Dmatrix_f y_c(t)$, and let us define the allocator error as $e(t)=u_p(t)-y_c(t)$. Furthermore, consider $\theta=0$, i.e., $\M = \M_n$. Then using the definition of $\varphi(y_f)$ in \eqref{DZident}, the expression $e(t) = \left( \M \Dmatrix_f -\Imm\right) y_c(t) + \M \C_f x_f(t) + \M \varphi(y_f(t))$ is easily obtained. It is straightforward to see that the choice $\Dmatrix_f =\Mp$, $\C_f=\N$ leads to $e(t)=\M \varphi(y_f(t))$, therefore the error is null in the absence of saturation and uncertainties. Furthermore, by guaranteeing convergence of the extended vector $x(t)=\begin{bmatrix} x_p^{\T}(t) & x_c^{\T}(t) & x_f^{\T}(t) \end{bmatrix} ^{\T}$ to the origin, we always obtain $e^{*}=0$, where $e^{*}$ is the steady-state value of $e$. 
\end{remark}

\begin{remark}
The optimal solution to the cost function
\begin{equation}\label{eq:cost}
    \min\limits_{x_f} \Tm (y_f)=y_f^{\T}\W y_f \text{ subject to } y_f = \N x_f + \Mp y_c^{*},
\end{equation}
where $y_c^{*}$ is any controller output, 
is given by $x_f=-(\N^{\T}\W^{\T}\N)^{-1}\N^{\T}\W\Mp y_c^{*}$. 
That corresponds to the steady-state value of $x_f$ in \eqref{eq:alloc}. 
\end{remark}

\subsection{Closed-loop system and problem formulation}

By taking into account the definitions of $\mathcal{P}$, $\mathcal{C}$, $\mathcal{F}$, the definition of $\varphi(y_f)$ in \eqref{DZident} and the connection $u_p=\M sat(y_f)$, using~\eqref{eq:uncertainInfluence} and letting $\N$ in $ \reals ^{m_a \times n_f}$ be a basis for the Kernel of $\M_n$ (the ``nominal'' part of $\M(\theta)$) and $\Mp$ be the right pseudo-inverse of $\M_n$, the complete closed-loop system with $x(t)=\begin{bmatrix} x_p^{\T}(t) & x_c^{\T}(t) & x_f^{\T}(t) \end{bmatrix} ^{\T}$ in $\reals^{n}$, $n=n_p+n_c+n_f$, can be written as
\begin{equation} \label{eq:CLsystemUncertain}
\begin{cases}
    \Dot{x}(t) = (\mathcal{A}(\theta)+\Lm_f \K_f \xbar{\C}) x(t) + (\mathcal{B}(\theta)+\Lm \Ematrix) \varphi(y_f(t)) + \xbar{B}_w w(t), \\
    y_f(t) = \C x(t), 
\end{cases}
\end{equation}
where $\mathcal{A}(\theta) = \A+\xbar{\B}\M_u(\theta)\C = \sum_{i=1}^{n_{\alpha}} \alpha_i \A_i$ and $\mathcal{B}(\theta) = \B+\xbar{\B}\M_u(\theta) = \sum_{i=1}^{n_{\alpha}} \alpha_i \B_i$, where the $\A_i$ and $\B_i$ matrices are simply given by $\A_i = A + \xbar{\B}\M_i\C$ and $\B_i = \B + \xbar{\B}\M_i$, respectively, for $i \in \{1,...,n_{\alpha}\}$. Furthermore, $\Lm_f = \begin{bmatrix}
    0_{n_f \times n_p} & 0_{n_f \times n_c} & \Imm_{n_f}
\end{bmatrix}^{\T}$, $\xbar{\C}=\N^{\T}\W\C$, $\xbar{\B}_w = \begin{bmatrix}
    \B_w^{\top} & 0 & 0
\end{bmatrix}^{\top}$ and
\begin{align*}
{\A}&=\begin{bmatrix} 
\A_0 & 0 \\
0 & 0
\end{bmatrix}, \B= 
\xbar{\B}\M ,~~\xbar{\B} = \begin{bmatrix}
    \B_p \\ 0 
\end{bmatrix}, \Ematrix=\begin{bmatrix} 
\Ematrix_c  \\
\Ematrix_f \\
\end{bmatrix}, \Lm = \begin{bmatrix}
    \Lm_c & \Lm_f
\end{bmatrix}, \\
\Lm_c &= \begin{bmatrix}
    0_{n_c \times n_p} & \Imm_{n_c} & 0_{n_c \times n_f}
\end{bmatrix}^{\T}, 
\C =\begin{bmatrix} 
\Mp \Dmatrix_c \C_p & \Mp \C_c & \N  
\end{bmatrix}
\end{align*}
\noindent with $\A_0$ defined in \eqref{eq:A0}.
Due to the presence of the deadzone in the closed-loop dynamics (\ref{eq:CLsystem}) we need to characterize a suitable region of the state space in which the stability is ensured (see, for example, \cite{Tarbouriech_2011}). In general  the global asymptotic stability of the origin (that is for any initial condition $x(0) \in \reals^{n}$) does not hold except if the open loop has suitable properties of stability \cite{Sontag_1984}. Hence, the regional stability 
(that is, only for initial conditions in a neighborhood of the origin) has to be studied.
Since 
exact characterization of the basin of attraction of the origin remains an open problem, a challenging problem consists in providing an estimate of the basin of attraction
as accurate as possible.   

Furthermore, in order to ensure some level of performance to the allocator in terms of the total energy consumption of the actuators, we impose conditions that limit the energy of the signal $sat(y_f)$. With respect to \eqref{eq:control} and \eqref{eq:alloc}, the main objective of the chapter is to  co-design the dynamic allocation function, that is $\K_f$, and $\Ematrix_f$, and the controller anti-windup gain $\Ematrix_c$. However, differently from~\cite{AlvesLima2021}, we consider the presence of uncertainties $\theta$ in the influence matrix and that the system is affected by unknown disturbances $w$.

\section{Preliminary results}\label{sec:prelim}
In this section we recall some useful results dealing with the deadzone and the way to address the design problem in the particular case without perturbation and uncertainty.

Consider a matrix $\G\in \reals ^{m_a \times n}$,  and define the set
\begin{equation}\label{polySet}
         \Lap(\bar{u})=  \{ x \in \reals ^{n} ;  |\G_{(i)} x| \leq \xbar{u}_{(i)}, i = 1, ..., m_a  \},
\end{equation}
\noindent Then, nonlinearity $\varphi(y_f)$, with $y_f$ in (\ref{eq:CLsystem}), satisfies the following Lemma directly derived from Remark 7.4 p.289 
in \cite{Tarbouriech_2011}.
\begin{lem}\label{lemma:Sec}[Generalized sector condition] If $x$ belongs to set $\Lap(\bar{u})$, defined in (\ref{polySet}), then the deadzone nonlinearity $\varphi(y_f)$ satisfies the following inequality for any diagonal~matrix~$\Ss$~in~$\mathbb{S}_{m_a}^{+}$ 
\begin{equation}\label{SecBoun}
     \varphi^{\T}(y_f) \Ss^{-1} [\varphi(y_f)+\C x + \G x]\leq 0.
\end{equation}
\end{lem}
Another important result widely known in the literature is re-enunciated next (see, for example, \cite{Mauricio_2001}).  
\begin{lem}\label{lemma:Finsler}[Finsler's Lemma]
Consider $\zeta $ in $ \reals^{n}$, ${\Upsilon}={\Upsilon}^{\T} $ in $ \reals^{n \times n}$, and $\Gamma $ in $ \reals^{m \times n}$. The following facts are equivalent:
\renewcommand{\theenumi}{\roman{enumi}}%
\begin{enumerate}
    \item $\zeta^{\T} {\Upsilon} \zeta<0$, $\forall \zeta$ such that $\Gamma\zeta=0$, $\zeta \neq 0$.
    \item $\exists \mathfrak{I} $ in $ \reals^{n \times m}$ such that ${\Upsilon}+\mathfrak{I} \Gamma+\Gamma^{\T}\mathfrak{I}^{\T} \prec 0$.
\end{enumerate}
\end{lem}

For pedagogical purposes, we review the results from~\cite{AlvesLima2021} for the nominal and undisturbed ($w(t)=0$ and $\theta=0$) system given by
\begin{equation} \label{eq:CLsystem}
\begin{cases}
    \Dot{x}(t) = (\A+\Lm_f \K_f \xbar{\C}) x(t) + (\B+\Lm \Ematrix) \varphi(y_f(t)) \\
    y_f(t) = \C x(t) 
\end{cases}
\end{equation}

Before reviewing the results from~\cite{AlvesLima2021}, let us define the auxiliary matrix $\Psi$. 

\begin{equation}\label{eq:Psi}
    \Psi = \left[
\setlength\arraycolsep{2pt}
\begin{array}{c|c}
    \Psi_a & \Psi_b \\
    \hline
    \star & \Psi_c
\end{array}
\right] =
\left[
\setlength\arraycolsep{2pt}
\begin{array}{cc|cc}
-\xbar{\J}-\xbar{\J}^{\T} & \Psi_{12} & \Psi_{13} & 0 \\
   \star  & \Psi_{22}  & \Psi_{23} &  \xbar{\J} \C^{\T}\hspace{-0.1cm}\W^{\frac{1}{2}} \\
    \hline
\star  &  \star & -2\Ss & \Ss \W^{\frac{1}{2}}  \\
   \star & \star  &  \star &  -\gamma \Imm 
\end{array}
\right]
\end{equation}
with $\Psi_{12}=\xbar{\Pp} + \A \xbar{\J}^{\T} \hspace{-0.1cm}+ \Z  -\xbar{\J}$, $\Z = $diag$(0_{n_p+n_c},\xbar{\Ku}_f)$, $\Psi_{13} = \B \Ss + \Lm \Ku_e$, $\Psi_{22}=\text{He}\{\A\xbar{\J}^{\T}\hspace{-0.1cm}+\Z\}$, $\Psi_{23} = \Psi_{13}-\xbar{\G}^{\T}\hspace{-0.15cm}-\xbar{\J} \C^{\T}$,  and where $\xbar{\J}=\begin{bmatrix}
 \xbar{\C}^{\perp}\J_o^{\T} & \J_f^{\T}
\end{bmatrix}^{\T} \hspace{-0.2cm}$ in $\reals^{n \times n}$, $\xbar{\C}^{\perp} \hspace{-0.1cm}$ in $\reals^{n \times (n_p+n_c)}$ is a matrix such that $\xbar{\C}\hspace{0.05cm}\xbar{\C}^{\perp} = 0$. Furthermore, $\{\xbar{\Pp}$, ${\J_o}$, ${\J_f}$, $\xbar{\Ku}_f $, $\Ku_e $, $\xbar{\G} $, ${\Ss}$, $\gamma\}$ are variables that will be specified in the sequence.

\begin{theorem}\label{stabtheorem}
Assume the existence of matrices $\xbar{\Pp} $ in $\mathbb{S}_n^{+}$, ${\J_o}$ in $ \reals ^{(n_p+n_c) \times (n_p+n_c)}$, ${\J_f}$ in $ \reals ^{n_f \times n}$, $\xbar{\Ku}_f $ in $ \reals ^{n_f \times n_f}$, $\Ku_e $ in $ \reals ^{(n_c+n_f) \times m_a}$, $\xbar{\G} $ in $ \reals ^{m_a \times n}$, diagonal matrix ${\Ss}={\Ss}^{\T} $ in $\mathbb{S}_{m_a}^{+}$ and positive scalar $\gamma$ such that 
\begin{equation}\label{main_matineq}
\Psi  \prec 0,
\end{equation}
\begin{equation}
\label{eqRAS}
\begin{bmatrix}
\xbar{\Pp} & \xbar{\G}_{\left(i\right)}^\T \\
\star & \xbar{u}^2_{\left(i\right)} \\
\end{bmatrix} \succeq 0, \text{ for } i = 1, ..., m_a, 
\end{equation}
\noindent hold. Then, matrices $\Ematrix=\begin{bmatrix} \Ematrix_c^{\T}  &  \Ematrix_f ^{\T} \end{bmatrix}^{\T}  = \Ku_e \Ss^{-1}$, $\Ku_f=\xbar{\Ku}_f {\left(\xbar{\C}\J_f^{\T}\right)}^{-1} \hspace{-0.2cm}$ are such that:
\begin{enumerate}
    \item the nominal undisturbed closed-loop system \eqref{eq:CLsystem} is asymptotically stable in the ellipsoid $\varepsilon(\Pp,1)=\{ x $ in $ \reals^{n}; x^{\T} \Pp x \leq 1 \}$, with $\Pp=\J \xbar{\Pp} \J^{{\T}}$ and $\J = \xbar{\J}^{-1}$;
    \item the energy of the actuators usage signal is limited and given by
    \begin{equation*}
        \bigintssss_0^\infty sat(y_f(\tau))^{\T} \W sat(y_f(\tau)) d \tau \leq \gamma.
    \end{equation*}
\end{enumerate}
\end{theorem}

\begin{proof}
The proof can be found in~\cite{AlvesLima2021}.
\end{proof}

\section{Main results}\label{sec:main}

In this section, we present new results that generalize the results in~\cite{AlvesLima2021} to systems affected by disturbances $w$ and uncertainties $\theta$ in the influence matrix. For clarity reasons, these two cases are treated separately. However, treating these two cases together is straightforward by combining the results proposed in this section. 

\subsection{Design with external disturbances}\label{sec:disturb}

Considering the case $\M = \M_n$ (that is without uncertainty in the influence matrix), the closed-loop system affected by disturbances \textcolor{black}{$w \in  \mathcal{W}$} becomes
\begin{equation} \label{eq:CLsystemPertubed}
\begin{cases}
    \Dot{x}(t) = (\A+\Lm_f \K_f \xbar{\C}) x(t) + (\B+\Lm \Ematrix) \varphi(y_f(t)) + \xbar{B}_w w(t), \\
    y_f(t) = \C x(t), 
\end{cases}
\end{equation}
where $x(t)$ is the same augmented state vector as before. 

\begin{theorem} \label{disturbtheorem}
Assume the existence of matrices $\xbar{\Pp} $ in $\mathbb{S}_n^{+}$, ${\J_o}$ in $ \reals ^{(n_p+n_c) \times (n_p+n_c)}$, ${\J_f}$ in $ \reals ^{n_f \times n}$, $\xbar{\Ku}_f $ in $ \reals ^{n_f \times n_f}$, $\Ku_e $ in $ \reals ^{(n_c+n_f) \times m_a}$, $\xbar{\G} $ in $ \reals ^{m_a \times n}$, diagonal matrix ${\Ss}={\Ss}^{\T} $ in $\mathbb{S}_{m_a}^{+}$ and positive scalars $\gamma$ and $\mu$ such that 
\renewcommand{\arraystretch}{1.5}
\begin{equation}\label{main_disturb}
{\Psi}_w = \left[
\setlength\arraycolsep{2pt}
\begin{array}{cccc|c}
     & \phantom{A} & \Psi \phantom{A}  & \phantom{A} & \star \\
    \hline 
    \xbar{B}_w^{\top} & \xbar{B}_w^{\top} & 0 & 0 & -\R
\end{array}
\right]  \prec 0
\end{equation}
\begin{equation}
\label{eqRASdisturb}
\begin{bmatrix}
\xbar{\Pp} & \xbar{\G}_{\left(i\right)}^\T \\
\star & \mu \xbar{u}^2_{\left(i\right)} \\
\end{bmatrix} \succeq 0, \text{ for } i = 1, ..., m_a, 
\end{equation}
\begin{equation}\label{eq:muandsigma}
    \sigma - \mu \geq 0
\end{equation}
\noindent hold with $\Psi$ defined as in~\eqref{eq:Psi}. Then, matrices $\Ematrix=\begin{bmatrix} \Ematrix_c^{\T}  &  \Ematrix_f ^{\T} \end{bmatrix}^{\T}  = \Ku_e \Ss^{-1}$, $\Ku_f=\xbar{\Ku}_f {\left(\xbar{\C}\J_f^{\T}\right)}^{-1} \hspace{-0.2cm}$ are such that:
\begin{enumerate}
    \item \textcolor{black}{for $w=0$, the closed-loop system~\eqref{eq:CLsystemPertubed}} is asymptotically stable in the ellipsoid $\varepsilon(\Pp,\mu)=\{ x $ in $ \reals^{n}; x^{\T} \Pp x \leq \mu^{-1} \}$, with $\Pp=\J \xbar{\Pp} \J^{{\T}}$ and $\J = \xbar{\J}^{-1}$;
    \item the energy of the actuators usage signal is limited and given by
    \begin{equation*}
        \bigintssss_0^\infty sat(y_f(\tau))^{\T} \W sat(y_f(\tau)) d \tau \leq \gamma \mu^{-1}.
    \end{equation*}
    \item for any $w \in \mathcal{W}$ and initial condition $x(0) \in \varepsilon(\Pp,\beta)$, $0<\beta^{-1} \leq \mu^{-1}-\sigma^{-1}$, the trajectories of \textcolor{black}{the closed-loop system~\eqref{eq:CLsystemPertubed}} do not leave the ellipsoid $\varepsilon(\Pp,\mu)$. 
\end{enumerate}
\end{theorem}

\begin{proof}
In the case of energy-bounded disturbances $w$, our goal is to prove that the inequality $\Dot{\V}(x)-w^{\top} R w < 0$ holds, since by integration, it leads to the fact that  $V(x(T))-V(0)-\bigintssss_0^{T} w(\tau)^\top \R w (\tau) d \tau < 0$, $\forall T$, which ensures the satisfaction of items $1$ and $3$ in the theorem. 

Note first that the satisfaction of inequality (\ref{main_disturb}) means that matrix $\xbar{\J}$ is non-singular. Considering again a quadratic Lyapunov function $\V(x)=x^{\T} \Pp x$, with $\Pp \succ 0 $ in $\mathbb{S}_n^{+}$, the satisfaction of \eqref{eqRASdisturb} ensures the inclusion of the ellipsoid $\varepsilon(\Pp,\mu)$ in the  polyhedral set $\Lap(\bar{u})$, defined in (\ref{polySet}). Therefore, the satisfaction of relation \eqref{eqRASdisturb} means that  Lemma \ref{lemma:Sec} applies and one gets $-2\varphi^{\T}(y_f) \Ss^{-1} [\varphi(y_f)+\C x + \G x] \geq 0$, for any $x $ in $ \varepsilon(\Pp,\mu) \subseteq \Lap(\bar{u})$. Then for $x \in \varepsilon(\Pp,\mu) \subseteq \Lap(\bar{u})$, one gets $\Dot{\V}(x) -w^{\top} R w \leq \Dot{\V}(x) -w^{\top} R w-2\varphi^{\T}(y_f) \Ss^{-1} [\varphi(y_f)+\C x + \G x] \leq \Dot{\V}(x)-w^{\top} R w-2\varphi^{\T}(y_f) \Ss^{-1} [\varphi(y_f)+\C x + \G x]+\gamma^{-1}sat(y_f)^{\T} \W sat(y_f)$. Hence to obtain  $\Dot{\V}(x)-w^{\top} R w<0$,
it suffices that
\begin{equation} \label{ineq1disturb}
\begin{array}{c}
\Dot{\V}(x)-w^{\top} R w-2\varphi^{\T}(y_f) \Ss^{-1} [\varphi(y_f)+(\C + \G)x]\\
+\gamma^{-1}sat(y_f)^{\T} \W sat(y_f)<0,
\end{array}
\end{equation}
\noindent with $\Ss \succ 0$. 
Consequently, $\Dot{\V}(x)-w^{\top} R w+\gamma^{-1}sat(y_f)^{\T} \W sat(y_f)<0$ is also satisfied, which can be integrated resulting in
\begin{equation}\label{eq:erro_disturb}
   \gamma^{-1} \bigintssss_0^\infty \hspace{-0.1cm}sat(y_f(\tau))^{\T} \W sat(y_f(\tau)) d \tau  < \V(x(0)) +  \sigma^{-1} \leq \beta^{-1} + \sigma^{-1} \leq \mu^{-1},
\end{equation}
\noindent which leads to item ii) of Theorem \ref{stabtheorem}. By using \eqref{DZident} and defining the augmented vector  $\zeta_w=\begin{bmatrix} \Dot{x}^\T  & x^\T  &  \varphi(y_f)^\T & w^\T \end{bmatrix} ^\T$, we can rewrite inequality \eqref{ineq1disturb} as $\zeta_w^{\T} \Upsilon_w \zeta_w<0$, with the matrix $\Upsilon_w$ given by $\Upsilon_w = diag(\Upsilon,-R)$, where $\Upsilon$ is given by 
$$ \begin{bmatrix} 
  0  & ~~~\Pp &  ~~~0 \\
  \star  &  ~~~\C^{\T} \W^{\frac{1}{2}} \gamma^{-1} \W^{\frac{1}{2}} \C  &  ~~~-(\G+\C)^{\T}\Ss^{-1}+ \C^{\T} \W^{\frac{1}{2}} \gamma^{-1} \W^{\frac{1}{2}} \\
  \star  &  ~~~\star & ~~~\W^{\frac{1}{2}} \gamma^{-1} \W^{\frac{1}{2}} -2\Ss^{-1}
  \end{bmatrix}.$$ From the closed-loop system dynamics, the relation $\Gamma_w\zeta_w=0$ holds for 
\begin{equation} \label{Bbardefinition_disturb}
\setlength\arraycolsep{5pt}
 \Gamma_w = 
  \begin{bmatrix} 
  -\Imm & \A+\Lm_f \Ku_f \xbar{\C} &  \B+\Lm\Ematrix & \xbar{B}_w \\
   \end{bmatrix}.
\end{equation}
From Lemma \ref{lemma:Finsler}, by considering $\mathfrak{I}_w = \begin{bmatrix}
    \J^{\T} & \J^{\T} & 0 & 0 \end{bmatrix}^{\T}$, we obtain the new condition $\xbar{\Psi}_w={\Upsilon}_w+\mathfrak{I}_w \Gamma_w+\Gamma_w^{\T}\mathfrak{I}^{\T}_w \prec 0$. By applying a Schur complement to $\xbar{\Psi}_w$, followed by congruence transformation with a permutation matrix allowing permuting the fourth and fifth columns and lines, respectively, followed by pre- and post-multiplication by diag$(\J^{-1},\J^{-1},\Ss,\Imm,\Imm)$ and its transpose, respectively, using the specific structure $\xbar{\J}=\begin{bmatrix}
 \xbar{\C}^{\perp}\J_o^{\T} & \J_f^{\T}
\end{bmatrix}^{\T}$, as in the proof of Theorem~\ref{stabtheorem}, and making changes of variable $\xbar{\J}=\J^{-1}$, $\xbar{\Pp}=\xbar{\J} \Pp \xbar{\J}^{\T}$, $\xbar{\G}=\G \xbar{\J}^{\T}$, $\Ku_e=\Ematrix \Ss$, $\xbar{\Ku}_f = {\Ku}_f \xbar{\C} \J_{f}^{\T}$, we obtain condition~\eqref{main_disturb}. 

Hence, it follows that if relations (\ref{main_disturb}) and \eqref{eqRASdisturb} are satisfied then \eqref{ineq1disturb} is also satisfied, implying that $\Dot{\V}(x)-w^{\T}Rw<0$, for any $x $ in $ \varepsilon(\Pp,\beta)$ and $w \in \mathcal{W}$. Then all the items of Theorem~\ref{disturbtheorem} are proven.
\end{proof}

\subsection{Robust design}\label{sec:robust}
For simplicity of presentation, let us consider the case with $\B_w=0$, i.e., the system is not affected by disturbances (the extension is trivial). The following theorem presents LMI conditions for robust co-design of the allocator and anti-windup in the presence of uncertainty in the influence matrix \textcolor{black}{while assuring stability and performance aspects to the closed-loop system~\eqref{eq:CLsystemUncertain} with $\xbar{\B}_w=0$}. 

\begin{theorem} \label{robusttheorem}
Assume the existence of $n_{\alpha}$ matrices $\xbar{\Pp}_i $ in $\mathbb{S}_n^{+}$, matrices ${\J_o}$ in $ \reals ^{(n_p+n_c) \times (n_p+n_c)}$, ${\J_f}$ in $ \reals ^{n_f \times n}$, $\xbar{\Ku}_f $ in $ \reals ^{n_f \times n_f}$, $\Ku_e $ in $ \reals ^{(n_c+n_f) \times m_a}$, $\xbar{\G} $ in $ \reals ^{m_a \times n}$, diagonal matrix ${\Ss}={\Ss}^{\T} $ in $\mathbb{S}_{m_a}^{+}$ and positive scalar $\gamma$ such that 
\begin{equation}\label{main_matineq:robust}
\Psi_{i} = \left[
\setlength\arraycolsep{2pt}
\begin{array}{c|c}
    \Psi_{a_{i}} & \Psi_{b_{i}} \\
    \hline
    \star & \Psi_c
\end{array}
\right] =
\left[
\setlength\arraycolsep{2pt}
\begin{array}{cc|cc}
-\xbar{\J}-\xbar{\J}^{\T} & \Psi_{12_{i}} & \Psi_{13_{i}} & 0 \\
   \star  & \Psi_{22_{i}}  & \Psi_{23_{i}} &  \xbar{\J} \C^{\T}\hspace{-0.1cm}\W^{\frac{1}{2}} \\
    \hline
\star  &  \star & -2\Ss & \Ss \W^{\frac{1}{2}}  \\
   \star & \star  &  \star &  -\gamma \Imm 
\end{array}
\right] \prec 0
\end{equation}
\begin{equation}
\label{eqRAS:robust}
\begin{bmatrix}
\xbar{\Pp}_{i} & \xbar{\G}_{\left(j\right)}^\T \\
\star & \xbar{u}^2_{\left(j\right)} \\
\end{bmatrix} \succeq 0, \text{ for } j = 1, ..., m_a, 
\end{equation}
\noindent hold for all $i \in \{1,...,n_{\alpha}\}$, with $\Psi_{12_{i}}=\xbar{\Pp}_{i} + \A_{i} \xbar{\J}^{\T} \hspace{-0.1cm}+ \Z  -\xbar{\J}$, $\Z = $diag$(0_{n_p+n_c},\xbar{\Ku}_f)$, $\Psi_{13_{i}} = \B_{i} \Ss + \Lm \Ku_e$, $\Psi_{22_{i}}=\text{He}\{\A_{i}\xbar{\J}^{\T}\hspace{-0.1cm}+\Z\}$, $\Psi_{23_{i}} = \Psi_{13_{i}}-\xbar{\G}^{\T}\hspace{-0.15cm}-\xbar{\J} \C^{\T}$,  and where $\xbar{\J}=\begin{bmatrix}
 \xbar{\C}^{\perp}\J_o^{\T} & \J_f^{\T}
\end{bmatrix}^{\T} \hspace{-0.2cm}$ in $\reals^{n \times n}$, $\xbar{\C}^{\perp} \hspace{-0.1cm}$ in $\reals^{n \times (n_p+n_c)}$ is a matrix such that $\xbar{\C}\hspace{0.05cm}\xbar{\C}^{\perp} = 0$. Then, matrices $\Ematrix=\begin{bmatrix} \Ematrix_c^{\T}  &  \Ematrix_f ^{\T} \end{bmatrix}^{\T}  = \Ku_e \Ss^{-1}$, $\Ku_f=\xbar{\Ku}_f {\left(\xbar{\C}\J_f^{\T}\right)}^{-1} \hspace{-0.2cm}$ are such that, for any $\theta \in \Theta$:
\begin{enumerate}
    \item the closed-loop system \eqref{eq:CLsystemUncertain} with $w=0$ is asymptotically stable in the ellipsoid $\varepsilon(\Pp(\theta),1)=\{ x $ in $ \reals^{n}; x^{\T} \Pp(\theta) x \leq 1 \}$, with $\Pp(\theta) = \sum_{i=1}^{n_\alpha} \alpha_i \Pp_i$, where $\Pp_i=\J \xbar{\Pp}_i \J^{{\T}}$ and $\J = \xbar{\J}^{-1}$;
    \item the energy of the actuators usage signal is limited and given by
    \begin{equation*}
        \bigintssss_0^\infty sat(y_f(\tau))^{\T} \W sat(y_f(\tau)) d \tau \leq \gamma.
    \end{equation*}
\end{enumerate}
\end{theorem}
\begin{proof}
    Note first that the satisfaction of inequalities~\eqref{main_matineq:robust} means that matrix $\xbar{\J}$ is non-singular. Consider then a quadratic Lyapunov function $\V(x)=x^{\T} \Pp(\theta) x$, $\Pp(\theta) = \sum_{i=1}^{n_\alpha} \alpha_i \Pp_i$, with $\Pp_i \succ 0 $ in $\mathbb{S}_n^{+}$. 

    Next, note that since $\alpha_i\geq0$ and $\sum \limits_{i=1}^{n_{\alpha}}  \alpha_i=1$ for $i=1,\dots,n_{\alpha}$, by multiplying all the terms of \eqref{main_matineq:robust} and~\eqref{eqRAS:robust} by \( \alpha_i\) and summing them up from \(i=1\) to \(i=N\), we obtain, respectively
    \begin{equation}
        \Psi(\theta) = \sum \limits_{i=1}^{n_{\alpha}} \alpha_i \Psi_{i} \prec 0, \text{ for all } \theta \in \Theta,
    \end{equation}
    and
\begin{equation}
\label{eqRAS:theta}
\begin{bmatrix}
\xbar{\Pp}(\theta) & \xbar{\G}_{\left(j\right)}^\T \\
\star & \xbar{u}^2_{\left(j\right)} \\
\end{bmatrix} \succeq 0, \text{ for all } \theta \in \Theta.
\end{equation}    
    The satisfaction of relation~\eqref{eqRAS:robust} thus ensures, for all $\theta \in \Theta$, the inclusion of the ellipsoid $\varepsilon(\Pp(\theta),1)=\{ x $ in $ \reals^{n}; x^{\T} \Pp(\theta) x \leq 1 \}$ in the set $\Lap(\bar{u})$ as defined in (\ref{polySet}) by using the changes of variables $\xbar{\G}=\G \xbar{\J}^{\T}$, $\Pp(\theta) =\J \xbar{\Pp}(\theta) \J^{{\T}}$ and $\J = \xbar{\J}^{-1}$. 
    
    Therefore, the satisfaction of relation~\eqref{eqRAS:robust} means that  Lemma~\ref{lemma:Sec} applies and one gets $-2\varphi^{\T}(y_f) \Ss^{-1} [\varphi(y_f)+\C x + \G x] \geq 0$, for any $x $ in $ \varepsilon(\Pp(\theta),1) \subseteq \Lap(\bar{u})$. 
    
    Then for $x \in \varepsilon(\Pp(\theta),1) \subseteq \Lap(\bar{u})$, one gets $\Dot{\V}(x) \leq \Dot{\V}(x)-2\varphi^{\T}(y_f) \Ss^{-1} [\varphi(y_f)+\C x - \G x] \leq \Dot{\V}(x)-2\varphi^{\T}(y_f) \Ss^{-1} [\varphi(y_f)+\C x + \G x]+\gamma^{-1}sat(y_f)^{\T} \W sat(y_f)$. 
    
    Hence to obtain  $\Dot{\V}(x)<0$ it suffices that
\begin{equation} \label{ineq1}
\begin{array}{c}
\Dot{\V}(x)-2\varphi^{\T}(y_f) \Ss^{-1} [\varphi(y_f)+(\C + \G)x]\\
+\gamma^{-1}sat(y_f)^{\T} \W sat(y_f)<0,
\end{array}
\end{equation}
\noindent with $\Ss \succ 0$. 
Consequently, $\Dot{\V}(x)+\gamma^{-1}sat(y_f)^{\T} \W sat(y_f)<0$ is also satisfied, which can be integrated resulting in
\begin{equation}\label{eq:erro}
   \gamma^{-1} \bigintssss_0^\infty \hspace{-0.1cm}sat(y_f(\tau))^{\T} \W sat(y_f(\tau)) d \tau  < \V(x(0)) \leq 1,
\end{equation}
\noindent which leads to item 2 of Theorem~\ref{robusttheorem}.

By using \eqref{DZident} and defining the augmented vector  $\zeta=\begin{bmatrix} \Dot{x}^\T \hspace{-0.1cm} & \hspace{-0.1cm} x^\T \hspace{-0.1cm} & \hspace{-0.1cm} \varphi(y_f)^\T \end{bmatrix} ^\T$, we can rewrite inequality \eqref{ineq1} as $\zeta^{\T} \Upsilon((\theta) \zeta<0$, with the matrix $\Upsilon(\theta)$ given by:
\begin{equation*} \label{Qdefinition}
 \begingroup 
\setlength\arraycolsep{0.5pt}
  \begin{bmatrix} 
  0  & ~~~\Pp(\theta) &  ~~~0 \\
  \star  &  ~~~\C^{\T} \W^{\frac{1}{2}} \gamma^{-1} \W^{\frac{1}{2}} \C  &  ~~~-(\G+\C)^{\T}\Ss^{-1}+ \C^{\T} \W^{\frac{1}{2}} \gamma^{-1} \W^{\frac{1}{2}} \\
  \star  &  ~~~\star & ~~~\W^{\frac{1}{2}} \gamma^{-1} \W^{\frac{1}{2}} -2\Ss^{-1}
  \end{bmatrix}.
  \endgroup
\end{equation*}
We also have that the relation $\Gamma(\theta)\zeta=0$ holds for 
\begin{equation} \label{Bbardefinition}
 \Gamma(\theta) = 
  \begin{bmatrix} 
  -\Imm & ~~~\A(\theta)+\Lm_f \Ku_f \xbar{\C} &  ~~~\B(\theta)+\Lm\Ematrix \\
   \end{bmatrix}.
\end{equation}
From Lemma \ref{lemma:Finsler}, by considering $\mathfrak{I} = \begin{bmatrix}
    \J^{\T} & \J^{\T} & 0 \end{bmatrix}^{\T}$, we obtain the new condition $\xbar{\Psi}(\theta)={\Upsilon}+\mathfrak{I} \Gamma(\theta)+\Gamma^{\T}(\theta)\mathfrak{I}^{\T} \prec 0$. By applying a Schur complement to $\xbar{\Psi}(\theta)$, followed by pre- and post-multiplying by diag$(\J^{-1},\J^{-1},\Ss,\Imm)$ and its transpose, respectively, and making changes of variable $\xbar{\J}=\J^{-1}$, $\xbar{\Pp}(\theta)=\xbar{\J} \Pp(\theta) \xbar{\J}^{\T}$, $\xbar{\G}=\G \xbar{\J}^{\T}$, $\Ku_e=\Ematrix \Ss$ we obtain the equivalent condition
\begin{equation*} \label{main_proof}
 \begingroup 
\setlength\arraycolsep{5pt}
  \begin{bmatrix}
   -\xbar{\J}-\xbar{\J}^{\T} & \xbar{\Pp} + (\A(\theta) + \Lm_f {\Ku}_f \xbar{\C}) \xbar{\J}^{\T} -\xbar{\J} & \B(\theta) \Ss + \Lm \Ku_e & 0 \\
   \star  & \text{He}\{(\A(\theta) + \Lm_f {\Ku}_f \xbar{\C}) \xbar{\J}^{\T}\}  & \B(\theta) \Ss + \Lm \Ku_e-\xbar{\G}^{\T}\hspace{-0.15cm}-\xbar{\J} \C^{\T}&  \xbar{\J} \C^{\T}\W^{\frac{1}{2}} \\
   \star  &  \star & -2\Ss & \Ss \W^{\frac{1}{2}}  \\
   \star & \star  &  \star &  -\gamma \Imm 
   \end{bmatrix} \endgroup \prec 0,
\end{equation*}
\noindent Note that $ \xbar{\C}^{\perp}$ in $\reals^{n \times (n_p+n_c)}$ is the orthogonal complement of $\xbar{\C}^{\T}$ (i.e. $\xbar{\C}~\xbar{\C}^{\perp}=0$ with $rank(\xbar{\C}) = n_f$) such that $\begin{bmatrix}
\xbar{\C}^{\perp} & \xbar{\C}^{\T}
\end{bmatrix}$ is square and nonsingular. Thanks to this, the specific structure $\xbar{\J}=\begin{bmatrix}
 \xbar{\C}^{\perp}\J_o^{\T} & \J_f^{\T}
\end{bmatrix}^{\T}$, with ${\J_o}$ in $ \reals ^{(n_p+n_c) \times (n_p+n_c)}$, ${\J_f}$ in $ \reals ^{n_f \times n}$ does not prevent the existence of $\J_o$ and $\J_f$ making $\xbar{\J}$ non singular.
Hence, such a structure for $\xbar{\J}$ allows to have $\Z = \Lm_f \begin{bmatrix}
    \xbar{\C}~\xbar{\C}^{\perp} {\J_o}^{\T}  &  {\Ku}_f \xbar{\C} \J_{f}^{\T}
\end{bmatrix}= diag(0_{n_p+n_c},\xbar{\Ku}_f)$ in the term $\Psi_{22}$, where the satisfaction of inequality \eqref{main_matineq} implies that $\xbar{\C}\J_f^{\T}$ is non-singular with a full row rank matrix $\J_{f}$ (i.e. $rank(\J_{f}) = n_f)$, allowing the computation of $\K_f$. Hence, it follows that if relations (\ref{main_matineq}) and \eqref{eqRAS} are satisfied then \eqref{ineq1} is also satisfied, or equivalently $\Dot{\V}(x)<0$, for any $x $ in $ \varepsilon(\Pp(\theta),1)$, for all $\theta \in \Theta$. Then the two items of Theorem~\ref{robusttheorem} are proven and the proof is completed.
\end{proof}

 \subsection{Remarks on feasibility issues}

Studying the feasibility properties of LMI conditions is an important subject. The following proposition about Theorem \ref{stabtheorem} was stated and proved in~\cite{AlvesLima2021}.
\begin{prop}\label{prop1}
LMI \eqref{main_matineq} in Theorem~\ref{stabtheorem} is always feasible.  
\end{prop}

Two new propositions can be stated about the feasibility of the new results in this chapter, i.e., Theorems~\ref{disturbtheorem} and~\ref{robusttheorem}.

\begin{prop}
    If $\R$ is considered a variable instead of given by the disturbance model,~\eqref{main_disturb} would remain an LMI, which is always feasible. 
\end{prop}
\begin{proof}
    The proof follows almost the same steps as the proof of Proposition~\ref{prop1}, with an additional Schur complement argument based on a large enough positive definite matrix~$\R$. Therefore, it is not repeated here. 
\end{proof}

\begin{prop}
  LMI condition~\eqref{main_matineq:robust} in Theorem~\ref{robusttheorem} is always feasible if, for each $i \in \{1,...,n_{\alpha}\}$, the corresponding matrix $A_i$ is Hurwitz.  
\end{prop}
\begin{proof}
    The proof follows the same steps as the proof of Proposition~\ref{prop1}, thus, it is not repeated here. 
\end{proof}

Additionally, the subsequent two remarks are worthy of note.
\begin{remark}[On the choice of matrix $\W$]
From Remark 3 and Theorems~1,~2,~and~3, it is clear that the entries of the matrix $\W$ are inversely proportional to the level of usage of the actuator. Although the user can specify any desired value $w_i>0$, one promising choice in the case the level of saturation of the different actuators is different 
is to make $w_i=\xbar{u}_{(i)}^{-2}$.
\end{remark}
\begin{remark}[Global stability]\label{rem:global}
In case the plant state matrix $\A_p$ is Hurwitz stable, global stability of the closed loop can be achieved and the design of $\Ku_f$, $\Ematrix_f$, $\Ematrix_c$ can also be realized by solving LMI~\eqref{main_matineq} with $\xbar{\G} = 0$. 
\end{remark}

\section{Optimization issues}\label{sec:optim}
From~\eqref{eq:erro} and~\eqref{eq:erro_disturb}, it becomes clear that minimization of $\gamma$ leads to minimization of the energy of $sat(y_f(t))$. Therefore, while solving the LMIs in Theorems \ref{stabtheorem} and~\ref{disturbtheorem} (or in Remark \ref{rem:global}), we can accomplish better results for the allocator by minimizing $\gamma$. In the case of Theorem~\ref{stabtheorem}, the maximization of the size of the ellipsoid $\varepsilon(\Pp,1)$ is also of interest (or of $\varepsilon(\Pp,\mu)$ for Theorem~\ref{disturbtheorem}). Therefore, a multi-objective optimization procedure applies. Consider a positive definite matrix $\Pp_0$ and the following matrix inequality
\begin{equation}\label{eq:maxP}
    \begin{bmatrix}
        \Pp_0 & ~~\Imm \\
        \star & ~~\xbar{\J}+\xbar{\J}^{\T} -\xbar{\Pp}
        \end{bmatrix} \succeq 0.
\end{equation}
Then, minimization of the trace of $\Pp_0$ indirectly leads to minimization of the trace of $\Pp$ and, therefore, to maximization of the ellipsoid $\varepsilon(\Pp,1)$. Consider weighting parameters $\rho_1$, $\rho_2$. Then the following optimization procedure takes place in case of Theorem \ref{stabtheorem} 
\begin{equation}\label{eq:opt1}
        \text{min } \left(\rho_1 \lambda + \rho_2 \gamma \right) \text{~subject to \eqref{main_matineq}, \eqref{eqRAS}, \eqref{eq:maxP}}, \Pp_0 \preceq \lambda \Imm
\end{equation}

In case global asymptotic stability is sought (Remark \ref{rem:global}), the following optimization procedure applies
\begin{equation}\label{eq:opt2}
        \text{min } \gamma \text{~subject to \eqref{main_matineq} with $\xbar{\G} = 0$}
\end{equation}

In the case of systems affected by external disturbances, i.e., of Theorem~\ref{disturbtheorem}, we can also try to minimize $\mu$ in order to maximize the set $\varepsilon(\Pp,\mu)$. Then, the following optimization procedure takes place,
\begin{equation}\label{eq:opt3}
        \text{min } \left(\rho_1 \lambda + \rho_2 \gamma + \rho_3 \mu \right) \text{~subject to \eqref{main_disturb}, \eqref{eqRASdisturb}, \eqref{eq:muandsigma}, \eqref{eq:maxP}}, \Pp_0 \preceq \lambda \Imm
\end{equation}
where $\rho_3$ is an additional weighting parameter. Finally, it is also worth noting that an iterative procedure using, for example, a line search can be used in order to minimize the value of $\sigma$, which equates to increasing the upper bound on the disturbance signals and thus enlarging the set of allowable disturbances $\mathcal{W}$. In the case of zero initial conditions, $x(0)=0$, it follows that $\sigma^{-1}=\mu^{-1}$, and one can estimate the maximization of the set $\mathcal{W}$ for when the system is at equilibrium.

\begin{remark}
    Optimization procedures similar to~\eqref{eq:opt1} can be applied to maximize the ellipsoid $\epsilon(\Pp(\theta),1)$, for all $\theta \in \Theta$, and minimize the energy of the actuators usage. In particular, considering weighting parameters $\rho_1,\rho_2$, the following optimization procedures takes place in the case of Theorem~\eqref{robusttheorem}
\begin{equation}\label{eq:opt8}
        \text{min } \left(\rho_1 \lambda + \rho_2 \gamma \right) \text{~subject to \eqref{main_matineq:robust}, \eqref{eqRAS:robust}}, \begin{bmatrix}
        \Pp_0 & ~~\Imm \\
        \star & ~~\xbar{\J}+\xbar{\J}^{\T} -\xbar{\Pp}_i
        \end{bmatrix} \succeq 0, \Pp_0 \preceq \lambda \Imm
\end{equation}
for all $i \in \{1,...,n_{\alpha}\}$. In case global asymptotic stability is sought, it suffices to minimize $\gamma$ subject to~\eqref{main_matineq:robust} with $\xbar{\G} = 0$.
\end{remark}

\section{Numerical results}\label{sec:simu}

In this section, we consider the satellite formation flying control problem from~\cite{Boada2013}, which also was studied in~\cite{AlvesLima2021}. The relative position between two satellites on a vertical axis is represented by the controlled output $y_p$. Given two satellites, the objective is to cancel the lateral position error between them in the $z-axis$ and 
the following model is considered
\begin{equation*}
   \left[
\begin{array}{c|c}
    \A_p & \B_p \\
    \hline
    \C_p & \Dmatrix_p
\end{array}
\right] =
\left[
\setlength\arraycolsep{5pt}\def\arraystretch{0.92}
\begin{array}{cc|cc}
 0 & 1 & 0 & 0 \\ 0 & 0 & m_1^{-1} & -m_2^{-1} \\
    \hline
1 & 0 & 0 & 0
\end{array}
\right],
\end{equation*}

\noindent with $m_1^{-1}$ and $m_2^{-1}$ the masses of the two satellites. The plant input is defined from forces $F_1$ and $F_2$  individually acting in each satellite and is given by $u_p=\left[\begin{array}{c}
     u_{p_1}  \\
     u_{p_2} 
\end{array}\right] =
\left[\begin{array}{c}
     F_{1}  \\
     F_{2} 
\end{array}\right]$.  Each satellite possesses 4 thrusters that jointly produce the force applied in each of them. We then consider the known part of the influence matrix given by $\M_n=\left[\begin{array}{cc}
     \M_{n_1} & 0  \\
     0 & \M_{n_2} 
\end{array}\right]$, with $\M_{n_1} = \M_{n_2} = \left[\begin{array}{cccc}
     1 & -1 & -1 & 1 
\end{array} \right]$. We assume that each thruster can produce a force between $0 ~mN$ and $100 ~mN$, therefore the saturation limits are not symmetric. We use the symmetrizing technique proposed in \cite{Boada2013}, which consists of substituting the asymmetric saturation by a symmetric one with limits $\xbar{u}_{i} = 50 ~mN, i = 1, \dots, 8$, followed by addition of the kernel symmetrizing vector $\xi = \xbar{u}$. We are then able to co-design the dynamic allocation device ($\mathcal{F}$ given by \eqref{eq:alloc}) and the anti-windup gain $\Ematrix_c$. Upon choosing $m_1=m_2=1000~kg$, a stabilizing LQG controller is designed using identity matrices for all the weights, as in~\cite{AlvesLima2021}. The resulting controller is given by 
\begin{equation*}
   \left[
\begin{array}{c|c}
    \A_c & \B_c \\
    \hline
    \C_c & \Dmatrix_c
\end{array}
\right] =
\left[
\setlength\arraycolsep{5pt}
\begin{array}{cc|c}
 -1.7321  & 1 & 1.7321 \\
 -1.0014 & -0.0532 & 1 \\
    \hline
-0.7071  & -26.6009 & 0 \\
0.7071  & 26.6009 & 0
\end{array}
\right].
\end{equation*}

We then compute $\Mp$=0.25diag$(\M_{n_1}^{\T},\M_{n_2}^{\T})$, $\N=\text{diag}(\N_1,\N_2)$, with $
    \N_1\hspace{-0.1cm}=\hspace{-0.1cm}\N_2\hspace{-0.1cm}=\hspace{-0.1cm}\begin{bmatrix}
1 & 1 & -1\\ 
& \Imm_3 &
\end{bmatrix}$.

\subsection{Example 1 - Perturbed system}

To illustrate the results from Section~\ref{sec:disturb}, we revisit the satellite example by adding a disturbance through a matrix $B_w = \begin{bmatrix}
    0 & m_1^{-1}
\end{bmatrix}^\T$. The disturbance model is given by $\R = 1$, $\sigma = 1$. We choose $\W$=diag$(100,1,\dots,1)$ to illustrate the allocator ability, which means that we want to penalize the use of the first actuator. By running optimisation problem~\eqref{eq:opt3}\footnote{To enlarge the region of stability in the direction of the first plant state, representing the distance between the satellites, we used a small modification in \eqref{eq:opt3} by substitution of $\Pp_0 \preceq \lambda \Imm$ by $\begin{bmatrix}
1 &0_{1 \times(n-1)}
\end{bmatrix} \Pp_0 \begin{bmatrix}
1 &0_{1 \times(n-1)}
\end{bmatrix}^{\T} \preceq \lambda$.} with $\rho_1=2$, $\rho_2=0.15$, and $\rho_3=1000$, we obtained the gains $\K_f$, $\Ematrix_c$, and $\Ematrix_f$ below

\begin{equation}\label{eq:Efex2}
   \left[
\begin{array}{c}
    \Ematrix_c \\  \hline
    \Ematrix_f
\end{array}
\right] =
\left[
\setlength\arraycolsep{5pt}\def\arraystretch{0.92}
\begin{array}{cccccccc}
  -0.1708 &  -0.0059 &  -0.0032  & -0.0053 &   0.0111 &  -0.0066 &  -0.0066  & 0.0066\\
   -0.0895  & -0.0118  & -0.0103 &   0.0057 &  -0.0024 &   0.0048  &  0.0048  & -0.0048\\
\hline
  -0.0480 &   -0.0103 &    0.1031  &  -0.0662  &   0.0747  &   0.1042   &  0.1042 &   -0.1042\\
   -0.0805  &   0.0057  &  -0.0661  &   0.0225 &   -0.0512  &  -0.0537  &  -0.0537  &   0.0537\\
    0.1277 &   -0.0018  &   0.0566 &  -0.0389 &   -0.0080  &   0.0722   &  0.0722 &   -0.0722\\
   -0.0757  &   0.0037  &   0.0790 &   -0.0408  &   0.0722 &    0.9615  &  -0.3662 &    0.3662\\
   -0.0757  &   0.0037  &   0.0790 &   -0.0408   &  0.0722 &   -0.3662  &   0.9615  &   0.3662\\
    0.0757  &  -0.0037  &  -0.0790 &    0.0408 &   -0.0722  &   0.3662  &   0.3662  &   0.9615\\
\end{array}
\right],
\end{equation}
\begin{equation}\label{eq:Afex2}
  \Ku_f =
\left[
\setlength\arraycolsep{2.3pt}\def\arraystretch{0.92}
\begin{array}{cccccc}
  -1.8960  &   0.9476  & -0.9437  &  0.0053  &  0.0053  & -0.0053 \\
    0.9103  & -1.8750  & -0.9090 &   0.0013  &  0.0013 &  -0.0013 \\
   -0.8741  & -0.8266  & -1.8904  & -0.0007 &  -0.0007  &  0.0007 \\
   -0.0020 &  -0.0009  &  0.0118 &  -1.6716   & 0.5010 &  -0.5010 \\
   -0.0020 &  -0.0009  &  0.0118   & 0.5010 &  -1.6716  & -0.5010 \\ 
    0.0020  &  0.0009 &  -0.0118 &  -0.5010  & -0.5010 & -1.6716
\end{array}
\right]
\end{equation}

We performed simulation with initial condition $x_p(0)=\left[ \begin{array}{cc}
     -0.18  & 0
\end{array}\right]^{\T}$, with $x_c(0) = 0$ and $x_f(0) =0$, and disturbance $w$ given by a function
\begin{equation*}
    w(t) = \begin{cases}
        0.1667, 0\leq t \leq 36 \\
        0, t > 36
    \end{cases}
\end{equation*}
The results are illustrated in Fig.~\ref{ex3:output}. To illustrate the ability of the dynamic allocator to redistribute the control effort according to the chosen $\W$ matrix, two cases are plotted: dynamic allocator ($\mathcal{F}$ defined in \eqref{eq:alloc}) plus anti-windup gain ($\Ematrix_c$) and static allocator ($\mathcal{F}=\Mp$) plus anti-windup gain ($\Ematrix_c$). Both strategies stabilize the system, however, it can be observed that the dynamic allocation successfully reduces the usage of the penalized actuator even in the presence of disturbances.

\begin{figure}[!hhh]
\center{\includegraphics[width=0.8\linewidth,angle=0]{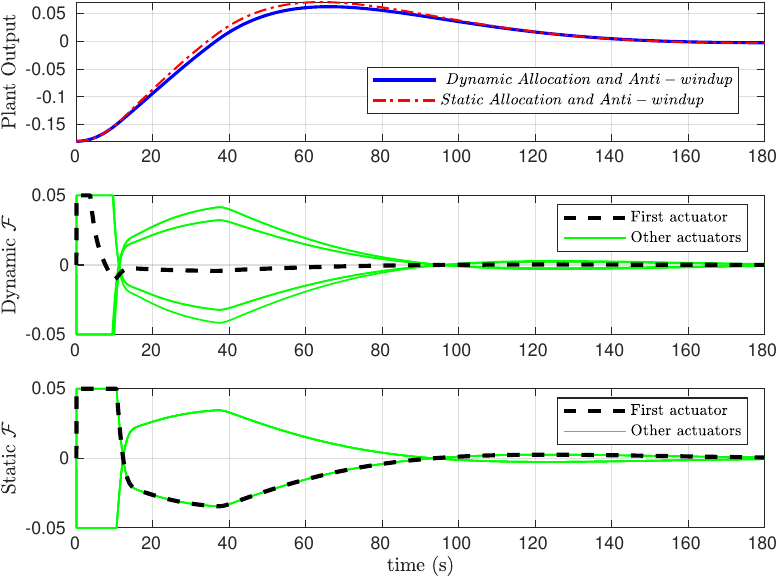}
\caption{Example 2: Output and actuators - external disturbance case.}
\label{ex3:output}}
\end{figure}

\subsection{Example 2 - Robust case}
To illustrate the results from Section~\ref{sec:robust}, we revisit again the satellite formation problem by considering the uncertain influence matrix given by $\M(\theta) = \M_n + \M_u(\theta)$, with $\M_n$ previously defined, and $\M_u(\theta) = \theta \begin{bmatrix}
    0 & 1 & 1 & -1 & 0 & 0 & 0 & 0 \\
      0 & 0 & 0 & 0  & 0 & 0 & 0 & 0
\end{bmatrix}$, $\theta \in [0.9,1] \subset \reals$. This influence matrix depicts a scenario in which three out of four thrusters on the first satellite are compromised, having a functional capacity ranging from zero to ten percent of their nominal capacity. The former case occurs when $\theta$ is equal to 1, while the latter occurs when $\theta$ is equal to 0.9.

By running optimisation problem~\eqref{eq:opt8} with $\rho_1=2$, $\rho_2=0.15$, we obtained the gains $\K_f$, $\Ematrix_c$, and $\Ematrix_f$ below

\begin{equation}\label{eq:Efex3}
   \left[
\begin{array}{c}
    \Ematrix_c \\  \hline
    \Ematrix_f
\end{array}
\right] =
\left[
\setlength\arraycolsep{5pt}\def\arraystretch{0.92}
\begin{array}{cccccccc}
 0.0019 &  -0.0000   & 0.0394  & -0.0193 &   0.0325   & -0.0411 &  -0.0411   & 0.0411 \\
   -0.0002 &  -0.0047  &  0.0142  & -0.0043 &   0.0118   & -0.0160  & -0.0160  &  0.0160\\
\hline
 1.2781 &   0.0144   & 0.1663 &  -0.0741 &  0.1006    & 0.1297   & 0.1297 &  -0.1297 \\
   -0.6243  & -0.0044 &  -0.0738  &  0.3141   & 0.2881   &-0.0918  & -0.0918 &   0.0918 \\
    0.7725  &  0.0088 &   0.0736   & 0.2114  &  0.3749    & 0.0720   & 0.0720 &  -0.0720 \\
   -0.9763  & -0.0119  &  0.0949 &  -0.0674 &   0.0721    & 0.9519 &  -0.3357  &  0.3357 \\
   -0.9763 &  -0.0119 &   0.0949 &  -0.0674 &  0.0721   & -0.3357   & 0.9519   & 0.3357 \\
    0.9763  &  0.0119 &  -0.0949 &   0.0674 &  -0.0721  &  0.3357  &  0.3357 &   0.9519
\end{array}
\right],
\end{equation}
\begin{equation}\label{eq:Afex3}
  \Ku_f =
\left[
\setlength\arraycolsep{2.3pt}\def\arraystretch{0.92}
\begin{array}{cccccc}
    -1.1684  &  0.6813  & -0.4766 &   0.0034  &  0.0034 &  -0.0034 \\
    0.7282   &-1.0438  & -0.3054  &  0.0249  &  0.0249  & -0.0249 \\
   -0.4528  & -0.3418  & -0.8017 &   0.0284  &  0.0284  & -0.0284 \\
   -0.0200   & 0.0792  &  0.0584  & -0.8628  &  0.1381  & -0.1381 \\
   -0.0200   & 0.0792  &  0.0584  &  0.1381  & -0.8628  & -0.1381 \\
    0.0200   &-0.0792 &  -0.0584  & -0.1381 &  -0.1381  & -0.8628
\end{array}
\right]
\end{equation}

We performed simulation with initial condition $x_p(0)=\left[ \begin{array}{cc}
     -0.25  & 0
\end{array}\right]^{\T}$, $x_c(0) = 0$, and $x_f(0) =0$ for both the nominal system with dynamic allocator designed using Theorem~1 from~\cite{AlvesLima2021}, and the uncertain case with $\theta = 1$ and co-design using optimisation problem~\eqref{eq:opt8}, i.e., with matrices given in~\eqref{eq:Efex3}-\eqref{eq:Afex3}. The results are illustrated in Fig.~\ref{ex3:robust}. One can see that although the uncertainty has affected the performance of the controller by slowing down the convergence of the output, the allocator is still effective in reducing the usage of the first actuator and minimizing the allocation error $e$.

\begin{figure}[!hhh]
\center{\includegraphics[width=0.8\linewidth,angle=0]{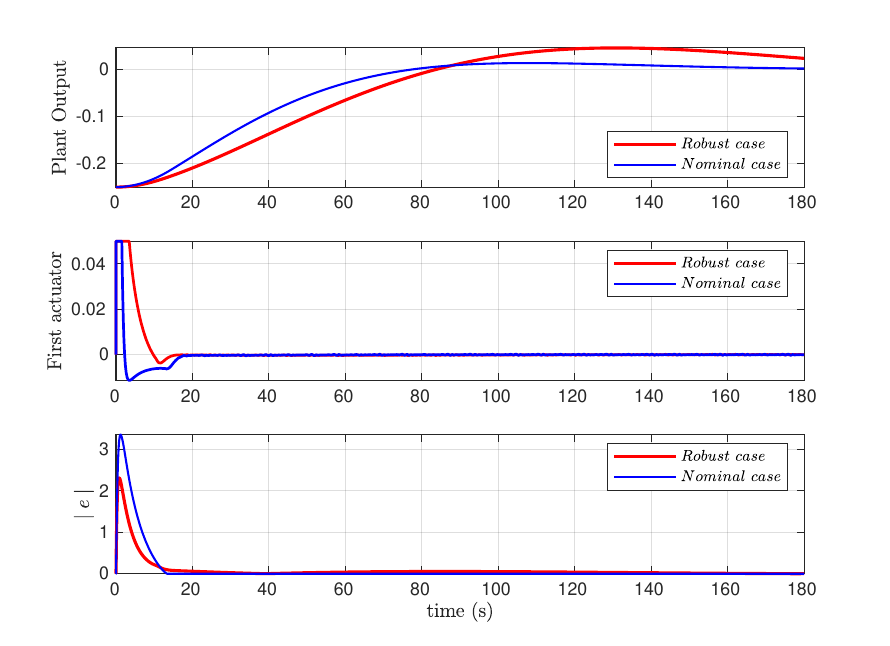}
\caption{Example 3: Output, the first actuator, and the allocation error for the robust case.}
\label{ex3:robust}}
\end{figure}

\section{Conclusion}\label{sec:conclu}
This chapter handled the co-design of dynamic allocation functions along with anti-windup gains to deal with over actuated/input redundant systems. The system  under consideration is subject to saturating actuators and possibly affected by additive bounded disturbance and the presence of uncertainty in the influence matrix. The proposed results can be seen as an extension of the ideas in \cite{ZACCARIAN_2009} and \cite{AlvesLima2021} to a more general scenario. Indeed, the chapter dealt with a much broader spectrum of cases and proposed optimization criteria that allow both to minimize or maximize several things as energy consumption in the actuators, estimations on the region of attraction, and admissible bounds on the disturbance affecting the system. The proposed results pave the way for future developments, as for example, the consideration of other nonlinearities affecting the actuator and event-triggered control. \textcolor{black}{It could be also interesting to further investigate other models for the external disturbance signals, for example by considering those generated by stochastic processes.}

\addcontentsline{toc}{section}{Appendix}

\bibliographystyle{plain}
\bibliography{refs}

\end{document}